\documentclass[11pt, reqno]{amsart}
\usepackage{amssymb,amsfonts,amsxtra,amsmath,enumerate}
\usepackage{dsfont,mathrsfs}
\usepackage[all]{xy}
\DeclareFontFamily{OT1}{rsfs}{}
\DeclareFontShape{OT1}{rsfs}{n}{it}{<-> rsfs10}{}
\DeclareMathAlphabet{\mathscr}{OT1}{rsfs}{n}{it}

\newtheorem{theorem}{Theorem}[section]
\newtheorem{lemma}[theorem]{Lemma}
\newtheorem{corol}[theorem]{Corollary}
\newtheorem{prop}[theorem]{Proposition}
\newtheorem{claim}[theorem]{Claim}

\newtheorem{conj}[theorem]{Conjecture}

{\theoremstyle{definition} \newtheorem{defin}[theorem]{Definition}}
{\theoremstyle{remark} \newtheorem{remark}[theorem]{Remark}
\newtheorem{example}[theorem]{Example}

\newcommand\restr[2]{{
  \left.\kern-\nulldelimiterspace 
  #1 
  \vphantom{\big|} 
  \right|_{#2} 
  }}
  
\title{K-theoretic defect in Chern class identity for a free divisor}
\author{Xia Liao}
\address{
KIAS
85 Hoegiro, Dongdaemun-gu
Seoul 02455
Republic of Korea
}
\email{liao@kias.re.kr}

\begin{document}
\maketitle

\begin{abstract}
Let $X$ be a nonsingular variety defined over an algebraically closed field of characteristic $0$, and $D$ be a free divisor. We study the motivic Chern class of $D$ in the Grothendieck group of coherent sheaves $G_0(X)$, and another class defined by the sheaf of logarithmic differentials along $D$. We give explicit calculations of the difference of these two classes when: $D$ is a divisor on a nonsingular surface; $D$ is a hyperplane arrangement whose affine cone is free.  
\end{abstract}

\section{introduction}

Let $X$ be a nonsingular complex variety, let $D$ be a free divisor, and let $U = X \smallsetminus D$ be the open complement of $D$ in $X$. The present paper aims to compute the following class
\begin{equation}\label{diff}
mC_{y}( [U\to X] )  - \sum_{p \geq 0}(\Omega^p_X( \log D)y^{p}) \otimes \mathscr{O}_X(-D) 
\end{equation}
in $G_0(X) \otimes \mathbb{Z}[y]$, where $G_0(X)$ denotes the Grothendieck ring of algebraic coherent sheaves on $X$, $mC_y$ is the motivic Chern class transformation (see section \ref{basic}), and $\Omega^p_X( \log D)$ is the sheaf of logarithmic $p$-forms along $D$ (see section \ref{freedivisor}). Here and in the rest of the paper, instead of writing $[\mathscr{F}]$ for the image of a coherent sheaf $\mathscr{F}$ in $G_0(X)$, we will simply write $\mathscr{F}$, for the ease of notation.

The primary motivation for our computation is the following comparison theorem \cite{liao}. 

\begin{theorem}\label{Chernlog}
Let $X$ be a nonsingular algebraic variety defined over an algebraically closed field of characteristic $0$, let $D$ be a free divisor of linear Jacobian type, and let $U = X \smallsetminus D$ be the open complement of $D$ in $X$. We have
\begin{equation}\label{formula}
c_*(1_U) = c(\textup{Der}_X(-\log D))\cap [X]
\end{equation}
in $A_*(X)$.
\end{theorem}

In this formula, the left side is the Chern class transformation of the indicator function on $U$ (which sometimes is also called the Chern-Schwartz-MacPherson class of $U$), the right side is the total Chern class of the sheaf of logarithmic derivations along $D$, and $A_*(X)$ is the Chow group of $X$. The formula is in particular true for locally quasi-homogeneous free divisors. See section \ref{freedivisor} for more explanation on this formula.

As will be explained in section \ref{freedivisor}, the difference class \eqref{diff} is a $K$-theoretic lift of the difference $c_*(1_U) - c(\textup{Der}_X(-\log D))\cap [X]$. So Theorem \ref{Chernlog} raised our expectation that \eqref{diff} might be zero for free divisors of linear Jacobian type. If the expectation were true, applying the Todd class transformation would give us an alternative prove of equation \eqref{formula} in $A_*(X) \otimes \mathbb{Q}$. However, it is unfortunately not the case except for SNC divisors. We will compute the class \eqref{diff} when $D$ is a curve on a nonsingular surface in section \ref{cur}, and when $D$ is a projective hyperplane arrangement whose affine cone is free in section \ref{hpa}.

Nevertheless, the difference class \eqref{diff} will represent a vanishing class for free divisors of linear Jacobian type, in the sense that after applying the generalised Todd class transformation $\textup{td}_{(1+y)*}$ to \eqref{diff} and specialising $y$ to $-1$, the class \eqref{diff} becomes $0$ in $A_*(X) \otimes \mathbb{Q}$. This is an easy consequence of theorem \ref{Chernlog} and our discussion in section \ref{freedivisor}.

I want to Thank J\"{o}rg Sch\"{u}rmann for his gracious help in various stages of my preparation of the paper. Some parts of the paper was conceived when I stayed in Oberwolfach in 2015 as Leibniz fellow. I also thank MFO for its hospitality.

\section{Hirzebruch classes of singular spaces}\label{basic}
For a nonsingular complete complex variety $X$, the Hirzebruch $\chi_y$-genus is a polynomial in a parameter $y$ with rational coefficients. Assigning $-1$ to $y$, the value of the polynomial equals the topological Euler characteristic of $X$. Similarly,  assigning $0$ to $y$, the value of the polynomial equals the arithmetic genus of $X$. And moreover, assigning $1$ to $y$, the value of the polynomial thereof equals the signature of the complex manifold $X$.

Recently, the motivic Chern class transformation and the Hirzebruch class transformation were introduced by Brasselet-Schurmann-Yokura. These characteristic class transformations generalise the Hirzebruch $\chi_y$-genus, are applicable to singular spaces, and unify the Chern class transformation theory of MacPherson-Schwartz, the Todd class transformation theory of Baum-Fulton-MacPherson, and the $L$-class transformation theory of Cappell-Shaneson. We will provide here the basic formulation of the Motivic Chern class transformation and Hirzebruch class transformation, in view of our applications in the rest of the paper. For a complete treatment of the theory, the readers are referred to \cite{MR2646988}.

\begin{defin}\cite{MR2059227}
Let $X$ be a complex variety. The relative Grothendieck group of varieties over $X$ is the quotient of the free abelian group generated by all the morphisms $Y \to X$, by the subgroup generated by all the additivity relations

\begin{equation*}
[Y \to X] = [Z \to  X] + [Y \smallsetminus Z \to X]
\end{equation*}  
where $Z$ is a closed subvariety of $Y$. We denote this group by $K_0(var\slash X)$.
\end{defin}

Notice that the construction of the relative Grothenieck group of varieties gives us a functor $K_0(var\slash -)$ in the category of complex varieties. This functor assigns to any $X$ the group $K_0(var\slash X)$, and to any morphism $X_1 \to X_2$, the group homomorphism defined by the natural composition 

\begin{equation*}
[Y \to X_1] \mapsto [Y \to X_1 \to X_2].
\end{equation*}

For any variety $X$,  $G_0(X)$ denotes the Grothendieck group of coherent sheaves on $X$. It is also a functor, but only for proper morphisms. The key of the theory of Hirzebruch class is the existence of the motivic Chern class transformation $mC_y$, from $K_0(var\slash -)$ to $G_0(-)\otimes \mathbb{Z}[y]$. The natural transformation $mC_y$ satisfies the normalisation condition that 

\begin{equation*}
mC_y([X\xrightarrow{id}X]) = \sum_{i=0}^{d} \Omega_X^{i} y^i
\end{equation*}
 for a smooth and purely $d$-dimensional variety $X$.

Let $A_*(-)$ be the Chow functor. Namely, for any $X$, $A_*(X)$ is the Chow group of $X$. The todd class transformation from $G_0(-)$ to $A_*(-)\otimes \mathbb{Q}$ in the singular Riemann-Roch theorem \cite{MR1644323} extends to a natural transformation from $G_0(-)\otimes \mathbb{Z}[y]$ to $A_*(-)\otimes \mathbb{Q}[y]$ by $\textup{td}_*\otimes id_{\mathbb{Z}[y]}$, for which we still write $\textup{td}_*$. The unnormalized Hirzebruch class transformation is then the composition of $mC_y$ and $\textup{td}_*$. To normalise it, we first expand an arbitrary class $\alpha$ in $A_*(X) \otimes \mathbb{Q}[y]$ by its dimension:

\begin{equation*}
\alpha = \sum_{i} \alpha_i
\end{equation*} 
where $\alpha_i \in A(X)_i \otimes \mathbb{Q}[y]$. We then define its normalisation to be 
\begin{equation*}
\sum_{i} \frac{1}{(1+y)^i} \cdot \alpha_i. 
\end{equation*}

We denote the normalised Hirzebruch class transformation by $T_{y*}$. The composition of $\textup{td}_*$ with the normalisation is the generalised Todd class transformation $\textup{td}_{(1+y)*}$, defined by Yokura, used in the singular Riemann-Roch theorem \cite{MR1677403}. Therefore, we have
\begin{equation*}
T_{y*} = \textup{td}_{(1+y)*} \circ mC_y 
\end{equation*}

Fixing a variety $X$, from the procedure we have described, it is clear that $T_{y*}$ takes values in $A_*(X) \otimes \mathbb{Q}[y, (1+y)^{-1}]$. However, one can show that $T_{y*}$ actually takes values in the smaller group $A_*(X) \otimes \mathbb{Q}[y] \subset A_*(X) \otimes \mathbb{Q}[y, (1+y)^{-1}]$. Therefore, it makes sense for us to specialise the $y$ value to $-1$, $0$, and $1$, as we have commented on Hirzebruch $\chi_y$ genus. The result is amazingly elegant. When $y= -1$, the natural transformation $T_y$ factors through the Chern class transformation $c_*$ of Schwartz and MacPherson; when $y=0$, it by construction factors through the Todd class transformation $\textup{td}_*$; when $y=1$ and $X$ is compact, $cl \circ T_{y*}$ likewise factors through the $L$-class transformation of Cappell-Shaneson, where $cl: A_*(X) \to H_{2*}(X)$ is the cycle map from the Chow group to the Borel-Moore homology. Since the rest of the paper concerns little with the factorisation of $T_{y*}$, we are gratified by stating them in the vaguest fashion, leaving the readers to find the precise statements in \cite{MR2646988}. The only fact about factorisation we will use in the sequel is that, in $A_*(X)$ we have
\begin{equation}\label{HirzChern}
T_{-1*}([U \to X]) = c_*(1_U)
\end{equation}
where $U \to X$ is an open embedding and $1_U$ is the function on $X$ such that $1_U(p) = 1$ if $p \in U$ and $1_U(p) = 0$ otherwise.

\section{sheaf of logarithmic derivations}\label{freedivisor}
Let $X$ be a nonsingular variety, let $D$ be a reduced divisor on $X$, and let $i: D \to X$ be the closed embedding. There is a natural morphism of sheaves on $X$:
\begin{equation*}
TX \to i_*\mathscr{O}_D(D)
\end{equation*}
which intuitively is the quotient map from the tangent bundle of $X$ to the normal bundle of $D$. When $D$ is singular, the morphism is not surjective. The cokernel of this morphism is $\mathscr{O}_{D^s}(D)$ where $D^s$ is the singular subscheme of $D$. We define the sheaf of logarithmic derivations along $D$ as the kernel of this morphism, and denote it by $\textup{Der}_X(-\log D)$. If $h$ is a local equation of $D$, then $\textup{Der}_X(-\log D)$ is locally the set of all derivations $\delta$ such that $\delta{h} = 0$. A logarithmic differential $p$-form $\omega$ along $D$ is a rational differential $p$-form such that $h\omega$ and $h\cdot d\omega$ has no poles on local charts. The sheaf of logarithmic $p$-forms along $D$ is denoted by $\Omega_X(\log D)$. It is known that $\textup{Der}_X(-\log D)$ is a reflexive coherent sheaf, dual to $\Omega_X^1(\log D)$. From the definition, we see that the sheaves of logarithmic forms fit into a logarithmic De Rham complex
\begin{equation*}
\Omega_X^{\bullet}(\log D): 0 \to \mathscr{O}_X \to \Omega_X^1(\log D) \to \ldots
\end{equation*}

We say that $D$ is a free divisor if $\textup{Der}_X(-\log D)$ is locally free. In this case, the rank of $\textup{Der}_X(-\log D)$ is the same as the rank of $TX$, which is the dimension of $X$. The morphism $\textup{Der}_X(-\log D) \to TX$ is an injective morphism of sheaves, but not a morphism of vector bundles. 

For free divisors, we have
\begin{equation*}
\Omega_X^p(\log D) \cong \Lambda^p(\Omega_X^1(\log D))
\end{equation*}

A basic reference to logarithmic derivations, logarithmic forms and free divisors is \cite{MR586450}.

\begin{example}\cite{MR586450}
\begin{enumerate}[(i)]
\item Any reduced curve on a nonsingular surface is a free divisor.
\item Let $f$ be a germ of a holomorphic function which has an isolated critical point at the origin. The discriminant in the parameter space of the universal unfolding $f$ is a free divisor germ at the origin.
\end{enumerate}
\end{example}

A divisor $D$ is locally quasi-homogeneous if at any point $p \in D$ we can find a local analytic coordinate chart such that $p$ is the origin and $D$ has a weighted homogeneous local equation. A divisor $D$ is of linear Jacobian type if the Jacobian ideal of $D$ defined by the equation of $D$ and all its partial derivatives (in a coordinate chart) is an ideal of linear type. Namely, the symmetric algebra of the ideal is isomorphic to the Rees algebra of the ideal. It is shown that a locally quasi-homogeneous free divisor is of linear Jacobian type \cite{MR1931962}. Theorem \ref{Chernlog} states that equation \eqref{formula} holds for free divisors of linear Jacobian type.

Recalling equation \eqref{HirzChern}, we can reformulate \eqref{formula} by 
\begin{equation*}
T_{-1*}([U \to X]) = c(\textup{Der}_X(-\log D))\cap [X].
\end{equation*}

This is the primary motivation for our study in this paper. For we naturally ask how is $T_{y}([U \to X])$, or even $mC_y([U \to X])$ related to $\textup{Der}_X(-\log D)$? J\"{o}rg Sch\"{u}rmann pointed out to us that the expected relation must be

\begin{equation}\label{hirz}
mC_{y}( [U\to X] )   = \sum_{p \geq 0}(\Omega^p_X( \log D)y^{p}) \otimes \mathscr{O}_X(-D) \quad \text{in $G_0(X) \otimes \mathbb{Z}[y]$}. 
\end{equation}

There are two justifications for the expectation. First, its truth implies a weaker version of equation \eqref{formula}.

\begin{claim}
For free divisors, formula \eqref{hirz} implies formula \eqref{formula} in $A_*(X) \otimes \mathbb{Q}$.
\end{claim}
 
\begin{proof}
Because $X$ is nonsingular, the Todd class tranformation $\textup{td}_*$ is the classical Grothendieck-Riemann-Roch transformation $\textup{ch}(-)\cdot \textup{td}(TX) \cap [X]$. We need to apply this transformation to the right side of \eqref{hirz}, normalise the result, and then plug in $y=-1$. Using the fact that $\Omega^p_X( \log D) \cong \Lambda^{p}\Omega^1_X( \log D)$ for free divisors, we readily get
\begin{equation*}
\textup{ch}\Big(\sum_{p \geq 0}(\Omega^p_X( \log D)y^{p}) \otimes \mathscr{O}_X(-D)\Big) \cdot \textup{td}(TX) = e^{-D}\prod_{i=1}^{d} \frac{(1+ye^{-\alpha_i})\beta_i}{1-e^{-\beta_i}} 
\end{equation*}
where the $\alpha_i$'s are the Chern roots of $\textup{Der}_X(-\log D)$, the $\beta_i$'s are the Chern roots of $TX$ and $d= \dim X$. To normalise this cohomology class, by the procedure described in section \ref{basic}, we only need to multiply $(1+y)^{i-d}$ by its component in $A^{i}(X) \otimes \mathbb{Q}[y]$. Therefore, the normalised class equals
\begin{equation*}
\frac{1}{(1+y)^d} \cdot e^{-(1+y)D} \prod_{i=1}^{d} \frac{(1+ye^{-\alpha_i(1+y)})\beta_i(1+y)}{1-e^{-\beta_i(1+y)}} = e^{-(1+y)D} \prod_{i=1}^{d} \frac{(1+ye^{-\alpha_i(1+y)})\beta_i}{1-e^{-\beta_i(1+y)}}
\end{equation*}
                  
Calculating the limit by L'Hospital's rule for $y \to -1$, we see that the expression is reduced to $\prod (1+\alpha_i)$, which is the total (cohomology) Chern class of $\textup{Der}_X(-\log D)$. The other Chern roots mysteriously disappear at the calculation of the limit.
\end{proof}

Secondly, equation \eqref{hirz} is true for simple normal crossing (SNC) divisors \cite{MR3417881}. In fact, we have an even stronger statement.
 
 \begin{prop}[\cite{MR2784747} exercise $3.10$]
 Let $D$ be a SNC divisor on a nonsingular complex variety $X$, we have a filtered quasi-isomorphism

\begin{equation}\label{snc} 
\underline{\Omega}^{\bullet}_{X,D} \cong \Omega_X^{\bullet}(\log D)\otimes \mathscr{O}_X(-D).
\end{equation}
\end{prop}

The left side of the quasi-isomorphism is the Du Bois complex of the pari $(X, D)$. It is equipped with the Hodge filtration. The complex $\Omega_X^{\bullet}(\log D)\otimes \mathscr{O}_X(-D)$ is equipped with the trivial filtration. Therefore, another way to express \eqref{snc} is 
\begin{equation*}
Gr_F^{p}(\underline{\Omega}^{\bullet}_{X,D})[p] \cong \Omega_X^{p}(\log D)\otimes \mathscr{O}_X(-D). 
\end{equation*}  

It is shown in \cite{MR3417881} section $2.1.2$ that 
\begin{equation}\label{dubois}
mC_y([U \to X]) = \sum_{p \geq 0}Gr_F^{p}(\underline{\Omega}^{\bullet}_{X,D})[p]\cdot y^p
\end{equation}
for any closed subvariety $D$ and its complement $U$ in $X$. (For a bounded complex of coherent sheaves $\mathscr{F}^{\bullet}$, its class in $G_0(X)$ is define to be $\sum_i(-1)^i[\mathscr{F}^i]$. With our usual abbreviation of brackets, here and in the future $Gr_F^{p}(\underline{\Omega}^{\bullet}_{X,D})[p] $ means the class of this complex in $G_0(X)$.) Combining \eqref{snc} and \eqref{dubois}, we see that \eqref{hirz} is true for SNC divisors.

Can equation \eqref{hirz} be true in grearter generality? In particular, if it were true for locally quasi-homogeneous free divisors, it would be a nice generalisation of theorem \ref{Chernlog}. Unfortunately, we will see this expectation fails even in very simple cases. In the next two sections, we will calculate the difference class \eqref{diff} in explicit forms when $D$ is a singular curve on a surface and when $D$ is a projective hyperplane arrangement whose affine cone is free. On the other hand, if equation \eqref{hirz} were true for a certain class of free divisors, it must at least be true for the parameter $y=0$, namely, 
\begin{equation*}
mC_0([U \to X]) = \mathscr{O}_X(-D).
\end{equation*}

Using \eqref{dubois}, this means that
\begin{equation}\label{dzero}
Gr_F^{0}(\underline{\Omega}^{\bullet}_{X,D}) =  \mathscr{O}_X(-D)
\end{equation}
in $G_0(X)$.

It is also known that for any $p$, the complex $Gr_F^{p}(\underline{\Omega}^{\bullet}_{X,D})[p]$ fits into an exact triangle (\cite{MR2784747} equation ($3.9.1$), or \cite{MR3417881} section $2.1.2$)
\begin{equation*}
\underline{\underline{\Omega}}_{X,D}^p \to \Omega_{X}^p \to i_*\underline{\underline{\Omega}}_{D}^p \xrightarrow{+1}
\end{equation*}
in $D^b_{coh}(X)$, where $\underline{\underline{\Omega}}_{X,D}^p$ stands for $Gr_F^{p}(\underline{\Omega}^{\bullet}_{X,D})[p]$. The meanings of other notations are that, $i: D \to X$ is the inclusion of $D$ into $X$ and $\underline{\underline{\Omega}}_{D}^p = Gr_F^{p}(\underline{\Omega}_{D}^{\bullet})[p]$ where $\underline{\Omega}_{D}^{\bullet}$ is the Du bois complex of $D$ equipped also with the Hodge filtration $F$. In particular, when $p=0$, we have the exact triangle
\begin{equation*}
Gr_F^{0}(\underline{\Omega}^{\bullet}_{X,D}) \to \mathscr{O}_X \to i_*\underline{\underline{\Omega}}_{D}^0 \xrightarrow{+1},
\end{equation*}
which implies
\begin{equation*}
\begin{split}
i_*\underline{\underline{\Omega}}_{D}^0 &=  \mathscr{O}_X - Gr_F^{0}(\underline{\Omega}^{\bullet}_{X,D}) \\
&= \mathscr{O}_X - \mathscr{O}_X(-D) \\
&= i_*\mathscr{O}_D
\end{split}
\end{equation*}
in $G_0(X)$.

\begin{defin}(\cite{MR2796408} definition $4.3$)
A variety $X$ has Du Bois singularities if the natural morphism $\mathscr{O}_X \to Gr_F^{0}(\underline{\Omega}^{\bullet}_{X})$ is a quasi-isomorphism.
\end{defin}

Even though we don't know how to characterise reduced divisors satisfying equation \eqref{dzero}, we see at least from previous definition that hypersurfaces having only Du Bois singularities are a natural set of objects satisfying equation \eqref{dzero}. Can we find reduced hypersurfaces which are free and have only Du Bois singularities but are not normal crossing? The following result indicates that our pursuit is most likely futile.
\begin{prop}
If a reduced hypersurface is free and have only Du Bois singularities, then it is normal crossing in codimension 1.  
\end{prop}
\begin{proof}
A complex analytic space is normal crossing in codimension $1$ if away from a codimension $2$ subset the space has only normal crossing singularities. Now the proof of the proposition follows from a sequence of easy implications found in the literatures.
\begin{enumerate}[(i)]
\item Du Bois singularities are seminormal (\cite{MR2796408} section $7$). 
\item For algebraic varieties defined over an algebraically closed field of characteristic $0$, seminormality is equivalent to weak normality (\cite{MR632650} proposition $2.3$). The definition of weak normality and seminormality are both given in \cite{MR632650} section $2$.
\item The definition for seminormality used in \cite{MR2796408} is different from the one used in \cite{MR632650}. Their equivalence can be found in \cite{MR618807} proposition $1.3$. 
\item For hypersurfaces, weak normality implies that away from a codimension $2$ analytic subset the variety has only multi-crossing singularities (\cite{MR0463484} theorem $2$). On hypersurfaces, these multi-crossing singularities can only appear at the transversal intersection of two smooth components. This follows easily from the description of the space $V_{(n,m)}$ in \cite{MR0463484} section $1$.
\end{enumerate}
\end{proof}
We finally state Faber's conjecture (\cite{MR3319556} Question $5$).
\begin{conj}
A free divisor has only normal crossing singularities if and only if it is normal crossing in codimension $1$.
\end{conj}
If the conjecture were true, we would have no chance of finding free and Du Bois divisors other than normal crossing divisors. The conjecture was proved for locally quasi-homogeneous free divisors \cite{sl}.

\section{motivic Chern classes of curves}\label{cur}
Let $X$ be a nonsingular complex surface, and let $D$ be a reduced curve on $X$. We first aim to write down an explicit expression of $mC_{y}( [U\to X]$. As the readers will see, the calculations are local in nature. Without loss of generality, we may assume $D$ has only one singularity at $x$.

Let $\tilde{D}$ be the normalisation of $D$. We have the following diagram of resolution of singularity.

\begin{displaymath}
\xymatrix{ E  \ar[r] \ar[d] & \tilde{D} \ar[d]^{\pi} \\
                 x  \ar[r]^{i} & D }
\end{displaymath}
where $E = \pi^{-1}(x)$.

In $K_0(var \slash D)$, we have
\begin{equation*}
\begin{split}
[D \to D] &= [x \to D] + [\tilde{D} \to D] - [E \to D] \\
              &= i_*[x \to x] + \pi_*[\tilde{D} \to \tilde{D}] - \pi_*[E \to E].
\end{split}
\end{equation*}

Consequently 
\begin{equation*}
\begin{split}
mC_y([D \to D]) &= i_*mC_y([x \to x]) + \pi_*mC_y([\tilde{D} \to \tilde{D}]) -\pi_*mC_y[E \to E] \quad\text{by functoriality of $mC_y$} \\
                          &= \mathscr{O}_x + \pi_*\mathscr{O}_{\tilde{D}} - \pi_*(\Omega_{\tilde{D}}^{1})y - \pi_*\mathscr{O}_E. \quad \text{by normalisation property of $mC_y$}
\end{split}
\end{equation*}
Note that because $\pi$ is a finite morphism, no higher direct image is needed in the expression above. We also have omitted, and will tacitly omit writing all push forward operators along closed embeddings.

We then have 
\begin{equation}\label{mc}
\begin{split}
mC_y([U \to X]) =& mC_y([X \to X]) - mC_y([D \to X]) \\
                          =& \sum_p \Omega_X^{p} y^p - mC_y([D \to X]) \\
                          =& (\mathscr{O}_X - \pi_*\mathscr{O}_{\tilde{D}} - \mathscr{O}_x + \pi_*\mathscr{O}_E) + (\Omega_X^1 - \pi_*\Omega_{\tilde{D}}^1)y + \Omega_X^2y^2
\end{split}
\end{equation}

We can also obtain this equation by using the Du Bois complex associated to the resolution diagram above. The Du Bois complex of $D$ is (\cite{MR2796408} equation ($4.2.5$))
\begin{equation*}
\underline{\Omega}_D^{\bullet}: 0 \to \pi_*\mathscr{O}_{\tilde{D}} \oplus \mathscr{O}_x \to \pi_*\Omega_{\tilde{D}}^1 \oplus \pi_*\mathscr{O}_E \to 0. 
\end{equation*}

There is a natural morphism of complexes $f: \Omega_X^{\bullet} \to \underline{\Omega}_D^{\bullet}$. The Du Bois complex of the pair $(X, D)$ is $cone(f)[-1]$, the mapping cone of $f$ shifted by $-1$ (\cite{MR2784747} definition $3.9$). More explicitly, 
\begin{equation*}
\underline{\Omega}_{X,D}^{\bullet}: 0 \to \mathscr{O}_X \to \Omega_X^1 \oplus \pi_*\mathscr{O}_{\tilde{D}} \oplus \mathscr{O}_x \to \Omega_X^2 \oplus \pi_*\Omega_{\tilde{D}}^1 \oplus \pi_*\mathscr{O}_E \to 0.
\end{equation*} 

The graded quotients are
\begin{equation*}
\begin{split}
Gr_F^0(\underline{\Omega}_{X,D}^{\bullet}) &: 0 \to \mathscr{O}_X \to \pi_*\mathscr{O}_{\tilde{D}} \oplus \mathscr{O}_x \to \pi_*\mathscr{O}_E \to 0 \\
Gr_F^1(\underline{\Omega}_{X,D}^{\bullet})[1] &: 0 \to \Omega_X^1 \to \pi_*\Omega_{\tilde{D}} \to 0 \\
Gr_F^2(\underline{\Omega}_{X,D}^{\bullet})[2] &: 0 \to \Omega_X^2 \to 0
\end{split}
\end{equation*}
where the first non zero terms of all complexes all sit in degree $0$. Using \eqref{dubois}, we again get \eqref{mc} the expression for $mC_y([U \to X])$.

Now the difference class \eqref{diff} can be written as 
\begin{equation*}
\begin{split}
& mC_{y}( [U\to X] )  - \sum_{p \geq 0}(\Omega^p_X( \log D)y^{p}) \otimes \mathscr{O}(-D) \\
=& (\mathscr{O}_X - \pi_*\mathscr{O}_{\tilde{D}} - \mathscr{O}_x + \pi_*\mathscr{O}_E - \mathscr{O}_X(-D)) + \\
& \big(\Omega_X^1 - \pi_*\Omega_{\tilde{D}}^1 - \Omega_X^1(\log D)(-D)\big)y + \\ 
& \big(\Omega_X^2 - \Omega_X^2(\log D)(-D)\big)y^2.
\end{split}
\end{equation*}

We have 
\begin{equation*}
\begin{split}
\mathscr{O}_X - \pi_*\mathscr{O}_{\tilde{D}} - \mathscr{O}_x + \pi_*\mathscr{O}_E - \mathscr{O}_X(-D)
=& \mathscr{O}_{D} - \pi_*\mathscr{O}_{\tilde{D}} - \mathscr{O}_x + \pi_*\mathscr{O}_E \\
=& (-\delta + i -1)\mathscr{O}_x
\end{split}
\end{equation*}
where $\delta = \text{length}(\mathscr{O}_{\tilde{D}}\slash \mathscr{O}_D)$ is the $\delta$-invariant of the singularity $x$ and $i = \text{length}(\mathscr{O}_E) = \# (\pi^{-1}(x))$ is the number of the branches ramified at $x$.

Quite generally, for a free divisor $D$ on a nonsingular $X$ of dimension $d$, we have (see section \ref{freedivisor})
\begin{equation*}
\Omega_X^d(D) \cong \Omega_X^d(\log D).
\end{equation*} 
Because any reduced divisor on a surface is free, we see that the coefficient for the $y^2$ is $0$.

To calculate the coefficient for $y$, we use the logarithmic residue sequence
\begin{equation*}
0 \to \Omega_X^1 \to \Omega_X^1(\log D) \to R_D \to 0
\end{equation*}
where $R_D$ is the logarithmic residue, and the standard sequences
\begin{equation*}
0 \to \Omega_X^1(-D) \to \Omega_X^1 \to \Omega_X^1 \otimes \mathscr{O}_D \to 0,
\end{equation*}
\begin{equation*}
0 \to \mathscr{O}_D(-D) \to \Omega_X^1 \otimes \mathscr{O}_D \to \Omega_D^1 \to 0.
\end{equation*}

Note that for any closed embedding $Y \to X$ of complex analytic spaces defined by a coherent sheaf $\mathscr{I}$, $\mathscr{I}/\mathscr{I}^2 \to \Omega^{1}_{X} \otimes \mathscr{O}_Y$ is in general not injective. Here the injectivity of $\mathscr{O}_D(-D) \to \Omega_X^1 \otimes \mathscr{O}_D$ is due to the isolated singularities on $D$, and can be verified by a local computation as follows. Let $f \in \mathscr{O}_{X,x}$ be a local equation of $D$ in the neighbourhood of the singular point $x$, and let $z_1,z_2$ be local analytic coordinates on $X$. Locally, the morphism $\mathscr{O}_D(-D) \to \Omega_X^1 \otimes \mathscr{O}_D$ takes the form $g \mapsto (g\partial_1f, g\partial_2f)$ where $g \in \mathscr{O}_{D,x}$. Let $\tilde{g}$ be a lift of $g$ in $\mathscr{O}_{X,x}$. If $g\partial_1f = g\partial_2f = 0$ in $\mathscr{O}_{D,x}$, then both $\tilde{g}\partial_1f$ and $\tilde{g}\partial_2f$ are contained in $(f)$. Using the fact that $\mathscr{O}_{X,x}$ is a UFD, and denoting the gcd of $\partial_1f, \partial_2f, f$ by $h$, we conclude that $\tilde{g}h \in (f)$. However, if $h$ were not a unit in $\mathscr{O}_{X,x}$, $D$ would not have isolated singularities. Hence $\tilde{g} \in {f}$ and $g = 0$.

We then obtain
\begin{equation*}
\begin{split}
\Omega_X^1 - \pi_*\Omega_{\tilde{D}}^1 - \Omega_X^1(\log D)(-D) &= \Omega_X^1 - \pi_*\Omega_{\tilde{D}}^1 - \Omega_X^1(-D) - R_D(-D) \\
&= \Omega_X^1 \otimes \mathscr{O}_D - \pi_*\Omega_{\tilde{D}}^1 - R_D(-D) \\
&= (\Omega_D^1 - \pi_*\Omega_{\tilde{D}}^1) + \big(\mathscr{O}_D(-D) - R_D(-D)\big)
\end{split}
\end{equation*}

There are natural morphisms $f: \Omega_D^1 \to \pi_*\Omega_{\tilde{D}}^1$ and $\mathscr{O}_D \to R_D$, which are isomorphisms away from the singular point $x$. By \cite{MR3260143} corollary 3.6 (or \cite{MR3319132} proposition $4.10$), $\text{dim}_{\mathbb{C}}(R_D\slash \mathscr{O}_D)=\tau$, the Tjurina number of the singularity $x$. Let us temporarily write $\alpha$ for the number $\text{dim}_{\mathbb{C}}(\text{ker}(f))-\text{dim}_{\mathbb{C}}(\text{coker}(f))$. Hence the coefficient for $y$ is $(-\tau + \alpha)\mathscr{O}_x \in G_0(X)$. 

Putting these together, we now prove the following 
\begin{theorem}
 Let $X$ be a nonsingular complex surface, let $D$ be a reduced divisor on $X$, and let $U = X \smallsetminus D$ be the open complement of $D$ in $X$. Then the different class \eqref{diff} has the expression
\begin{equation*}
\sum_{x}\Big((-\delta_x + i_x -1)\mathscr{O}_x + (-\tau_x + \delta_x)\mathscr{O}_x \cdot y\Big)
\end{equation*}
in $G_0(X) \otimes \mathbb{Z}[y]$, where the sum is taken over all singular points on $D$.
\end{theorem}

\begin{proof}
By what we have discussed above, the difference class \eqref{diff} is 
\begin{equation*}
\sum_{x}\Big((-\delta_x + i_x -1)\mathscr{O}_x + (-\tau_x + \alpha_x)\mathscr{O}_x \cdot y\Big)
\end{equation*}
where $\alpha_x= \text{dim}_{\mathbb{C}}(\text{ker}(f_x))-\text{dim}_{\mathbb{C}}(\text{coker}(f_x))$ and $f$ is the natural map $\Omega_D^1 \to \pi_*\Omega_{\tilde{D}}^1$. So what truly requires proving is the equality $\alpha_x = \delta_x$. To do this, let us compute the difference $c_*(1_U) - c(\textup{Der}_X(-\log D))$ by applying the generalised Todd class transformation $\textup{td}_{(1+y)*}$ to \eqref{diff} and substituting $-1$ for the parameter $y$. Because 
\begin{equation*}
\textup{ch}(\mathscr{O}_x) \cdot \textup{td}(TX) \cap [X] = [x]
\end{equation*} 
by GRR, the result before normalisation is 
\begin{equation*}
\sum_{x}\Big((-\delta_x + i_x -1) + (-\tau_x + \alpha_x) \cdot y\Big)[x].
\end{equation*}
Since this class is in $A_0(X) \otimes \mathbb{Q}[y]$, normalisation has no effect on it (cf. section \ref{basic}). Consequently we can directly plug in $-1$ for $y$ in this expression, and get
\begin{equation*}
c_*(1_U) - c(\textup{Der}_X(-\log D)) = \sum_{x}(\tau_x - \delta_x - \alpha_x + i_x -1)[x].
\end{equation*}

On the other hand, by \cite{MR2928937} corollary 3.2, 
\begin{equation*}
c_*(1_U) - c(\textup{Der}_X(-\log D)) = \sum_x (\tau_x - \mu_x)[x].
\end{equation*}

Comparing these equations, we get $\mu_x = \delta_x + \alpha_x -i_x +1$. With the help of Milnor's fomula $\mu_x = 2\delta_x -i_x + 1$, we conclude that $\alpha_x = \delta_x$.
\end{proof}

\begin{remark}
\begin{enumerate}[(i)]

\item Let the parameter $y$ be 0. We see that
\begin{equation*}
mC_0([U \to X]) - \mathscr{O}_X(-D) = Gr_F^{0}(\underline{\Omega}^{\bullet}_{X,D}) - \mathscr{O}_X(-D) = \sum_x (-\delta_x + i_x -1)\mathscr{O}_x.
\end{equation*}
The term $-\delta_x + i_x -1$ can be interpreted as the local difference of the geometric genus and the arithmetic genus of $D$, according to Hironaka's general genus formula (\cite{MR0090850} Theorem $2$).

\item It is desirable if one can show $\alpha_x = \delta_x$ by a direct computation in the local ring. It seems an easy question, but I can't find any reference for it. 

\end{enumerate}
\end{remark}

\section{motivic Chern classes of free hyperplane arrangements}\label{hpa}

In the section we consider hyperplane arrangements of projective spaces. We follow the notation used in \cite{MR3047491}. The fixed ambient nonsingular space $X$ is $\mathbb{P}^n$. We write $\mathscr{A}$ for the hyperplane arrangement (instead of $D$), and $M(\mathscr{A})$ for the complement of the hyperplane arrangement. We also let $\hat{\mathscr{A}}$ be the affine cone of $\mathscr{A}$ in $\mathbb{A}^{n+1}$. We assume moreover that $\hat{\mathscr{A}}$ is a free hyperplane arrangement in $\mathbb{A}^{n+1}$. This is a significantly stronger condition than that $\mathscr{A}$ is a free divisor in $\mathbb{P}^n$, for it indicates that $\textup{Der}_{\mathbb{P}^n}(-\log \mathscr{A})$ splits into a direct sum of line bundles. More precisely, let $h$ be the defining equation for the free hyperplane arrangement $\hat{\mathscr{A}}$, and let $\{e_1, \ldots, e_{n+1} \}$ be its exponents. We have the following exact sequence of $S$-modules ($S=k[x_1, \ldots, x_{n+1}]$)
\begin{equation*}
0 \to S(1-e_1)\oplus\ldots\oplus S(1-e_{n+1}) \to S(1)^{n+1} \to S\slash (h) \otimes S(d) \to S\slash J(h) \otimes S(d) \to 0
\end{equation*}
where $d = \textup{deg}(h)$ is the number of the hyperplanes in $\mathscr{A}$, $J(h) = (\partial_{x_1}h, \ldots, \partial_{x_n}h)$ is the Jacobian ideal of $h$, and $e_1 = 1$. The generator of $S(1-e_1)$ is mapped to the Euler vector field $\sum_ix_i\partial_{x_i}$ in $S(1)^{n+1}$.

Comparing this exact sequence with the exact sequence
\begin{equation*}
0 \to \textup{Der}_X(-\log D) \to TX \to \mathscr{O}_D(D) \to \mathscr{O}_{D^s}(D) \to 0 
\end{equation*}
appeared in the definition of $\textup{Der}_X(-\log D)$, we immediately see that 
\begin{equation}\label{splitder}
\textup{Der}_{\mathbb{P}^n}(-\log \mathscr{A}) \cong \mathscr{O}_{\mathbb{P}^n}(1-e_2) \oplus \ldots \oplus \mathscr{O}_{\mathbb{P}^n}(1-e_{n+1}).
\end{equation}

Let us drop the assumption that $\hat{\mathscr{A}}$ for a moment. The calculation of $mC_y([M(\mathscr{A}) \to \mathbb{P}^{n}])$ follows easily along the line of calculating the Chern-Schwarz-MacPherson class of $M(\mathscr{A})$ performed in \cite{MR3047491} (proof of theorem 3.1). There, it was shown that 
\begin{equation*}
1_{M(\mathscr{A})} = \sum_{x\in L(\hat{\mathscr{A})}}\mu (x) 1_{\underline{x}}
\end{equation*}
where $L(\hat{\mathscr{A})}$ denotes the intersection lattice of $\hat{\mathscr{A}}$, $x$ is any linear subspace of $\mathbb{A}^{n+1}$ belonging to the intersection lattice, $\mu (x)$ is the value of the M\"{o}bius function of the intersection lattice at $x$, and $\underline{x}$ is the corresponding projective space of $x$. This formula can be reinterpreted in our context as
\begin{equation*}
[M(\mathscr{A}) \to \mathbb{P}^n] = \sum_{x\in L(\hat{\mathscr{A})}}\mu (x)[\underline{x} \to \mathbb{P}^n].
\end{equation*}
Now, applying $mC_y$ to this equation, and using the normalisation property of $mC_y$, we obtain that 
\begin{equation*}
mC_y([M(\mathscr{A}) \to \mathbb{P}^n]) = \sum_{x\in L(\hat{\mathscr{A})}}\Big(\mu (x)\sum_{p \geq 0} \Omega_{\underline{x}}^py^p\Big)
\end{equation*}
in $G_0(\mathbb{P}^n) \otimes \mathbb{Z}[y]$.

\begin{lemma}
\begin{enumerate}[(i)]

\item
\begin{equation}\label{mcproj}
\sum_{p=0}^{n} \Omega_{\mathbb{P}^n}^p y^p = \frac{(1+\mathscr{O}_{\mathbb{P}^n}(-1)y)^{n+1}}{1+y}
\end{equation}
\item Let $i: \mathbb{P}^m \to \mathbb{P}^n$ be an inclusion of a linear subspace $\mathbb{P}^m$ into $\mathbb{P}^n$. Then in $G_0(\mathbb{P}^n)$ we have 
\begin{equation}
i_*\mathscr{O}_{\mathbb{P}^m}(k) = (1-\mathscr{O}_{\mathbb{P}^n}(-1))^{n-m} \cdot \mathscr{O}_{\mathbb{P}^n}(k)
\end{equation}
\item
\begin{equation}
i_*(\sum_{p=0}^{m} \Omega_{\mathbb{P}^m}^p y^p) = (1-\mathscr{O}_{\mathbb{P}^n}(-1))^{n-m} \cdot \frac{(1+\mathscr{O}_{\mathbb{P}^n}(-1)y)^{m+1}}{1+y}
\end{equation}

\end{enumerate}
\end{lemma}

\begin{proof}
\begin{enumerate}[(i)] 
\item This follows immediately from the $\lambda$-ring structure of $G_0({\mathbb{P}^{n}})$ and the Euler sequence 
\begin{equation*}
0 \to \Omega_{\mathbb{P}^n}^1 \to \mathscr{O}(-1)^{n+1} \to \mathscr{O} \to 0.
\end{equation*}
Alternatively, this sequence gives us the following relation 
\begin{equation*}
\Omega_{\mathbb{P}^n}^p + \Omega_{\mathbb{P}^n}^{p-1} =  \binom{n+1}{p} \mathscr{O}(-p),
\end{equation*}
as a result of the existence of a certain filtration of $\Lambda^{p}\big(\mathscr{O}(-1)^{n+1}\big)$(\cite{MR801033} page 102). Multiplying $y^p$ to both sides of the equation and summing over $p$, we get
\begin{equation*}
(1+y) \sum_{p \geq 0} \Omega_{\mathbb{P}^n}^p y^p = \sum_{p} \binom{n+1}{p}\mathscr{O}(-p)y^p,
\end{equation*}
which is equivalent to \eqref{mcproj}.

\item Tensoring $\mathscr{O}(k)$ over the Kozul resolution 
\begin{equation*}
0 \to \mathscr{O}_{\mathbb{P}^n}(-(n-m)) \to \ldots \to \mathscr{O}_{\mathbb{P}^n}(-i)^{\binom{n-m}{i}} \to \ldots \to \mathscr{O}_{\mathbb{P}^n} \to \mathscr{O}_{\mathbb{P}^m} \to 0
\end{equation*}
and taking its class in $G_0(\mathbb{P}^n)$, we get
\begin{equation*}
\begin{split}
\mathscr{O}_{\mathbb{P}^m}(k) &= \sum_{i} (-1)^i \binom{n-m}{i} \mathscr{O}_{\mathbb{P}^n}(-i) \cdot \mathscr{O}_{\mathbb{P}^n}(k) \\
&= (1-\mathscr{O}_{\mathbb{P}^n}(-1))^{n-m} \cdot \mathscr{O}_{\mathbb{P}^n}(k).
\end{split}
\end{equation*}

\item According to $(\textup{i})$, we have 
\begin{equation*}
(1+y) \sum_{p=0}^{m} \Omega_{\mathbb{P}^m}^p y^p = \sum_{p} \binom{m+1}{p}\mathscr{O}_{\mathbb{P}^m}(-p)y^p
\end{equation*}
in $G_0({\mathbb{P}^{m}})$. Applying $i_*$ to this equation and using $(\textup{ii})$ yield the result.
\end{enumerate}
\end{proof}

Continue with the calculation of $mC_y([M(\mathscr{A}) \to \mathbb{P}^n])$. With the help of the lemma, we get 
\begin{equation*}
\begin{split}
mC_y([M(\mathscr{A}) \to \mathbb{P}^n]) &= \sum_{x\in L(\hat{\mathscr{A})}} \mu (x)\frac{(\mathscr{O}_{\underline{x}}+\mathscr{O}_{\underline{x}}(-1)y)^{\textup{dim}(x)}}{1+y} \\
&= \sum_{x\in L(\hat{\mathscr{A})}} \mu (x) \frac{(1-\mathscr{O}_{\mathbb{P}^n}(-1))^{n-\textup{dim}(x) +1}(1+\mathscr{O}_{\mathbb{P}^n}(-1)y)^{\textup{dim}(x)}}{1+y} \\
&= \frac{(1-\mathscr{O}_{\mathbb{P}^n}(-1))^{n+1}}{1+y}  \sum_{x\in L(\hat{\mathscr{A})}} \mu (x) \Big(\frac{1+\mathscr{O}_{\mathbb{P}^n}(-1)y}{1-\mathscr{O}_{\mathbb{P}^n}(-1)}\Big)^{\textup{dim}x}
\end{split}
\end{equation*} 

Recall that the characteristic polynomial $\chi_{\hat{\mathscr{A}}}(t)$ of the intersection lattice $L(\hat{\mathscr{A})}$ is defined by 
\begin{equation*}
\chi_{\hat{\mathscr{A}}}(t) =  \sum_{x\in L(\hat{\mathscr{A})}} \mu (x)t^{\textup{dim}(x)}.
\end{equation*}

Using the characteristic polynomial, we get the following expression of $mC_y([M(\mathscr{A}) \to \mathbb{P}^n])$.
\begin{theorem}
Let $\mathscr{A}$ be a hyperplane arrangement in $\mathbb{P}^n$, and let $\chi_{\hat{\mathscr{A}}}$ be the characteristic polynomial of $\hat{\mathscr{A}}$. We have
\begin{equation}
mC_y([M(\mathscr{A}) \to \mathbb{P}^n]) = \frac{(1-\mathscr{O}_{\mathbb{P}^n}(-1))^{n+1}}{1+y} \chi_{\hat{\mathscr{A}}}\Big(\frac{1+\mathscr{O}_{\mathbb{P}^n}(-1)y}{1-\mathscr{O}_{\mathbb{P}^n}(-1)}\Big)
\end{equation}
\end{theorem}

This result is true regardless of the freeness of $\mathscr{A}$ (or even $\hat{\mathscr{A}}$), However, when $\hat{\mathscr{A}}$ is a free arrangement with exponent $\{e_1, \ldots, e_{n+1} \}$, Terao's factorisation theorem (\cite{MR1217488} theorem $4.61$) tells us that 
\begin{equation*}
\chi_{\hat{\mathscr{A}}}(t) = (t - e_1) \ldots (t - e_{n+1}).
\end{equation*}

This factorisation implies that

\begin{corol}
Let $\mathscr{A}$ be a hyperplane arrangement in $\mathbb{P}^n$ whose affine cone $\hat{\mathscr{A}}$ is free with exponents $\{e_1, \ldots, e_{n+1} \}$. We have
\begin{equation}\label{mchyper}
mC_y([M(\mathscr{A}) \to \mathbb{P}^n]) = \frac{\displaystyle\prod_{i=1}^{n+1}\Big(1-e_i + (e_i+y)\mathscr{O}_{\mathbb{P}^n}(-1)\Big)}{1+y}.
\end{equation}
\end{corol}

On the other hand, the splitting \eqref{splitder} of $\textup{Der}_{\mathbb{P}^n}(-\log D)$ when $\hat{\mathscr{A}}$ is free implies that 
\begin{equation*}
\Omega_{\mathbb{P}^n}^1(\log D) \cong  \mathscr{O}_{\mathbb{P}^n}(e_2-1) \oplus \ldots \oplus \mathscr{O}_{\mathbb{P}^n}(e_{n+1}-1),
\end{equation*}
and by the $\lambda$-ring structure of $G_0(\mathbb{P}^n)$ we get
\begin{equation*}
\sum_{p\geq 0} \Omega_{\mathbb{P}^n}^p(\log D) y^p = \prod_{i=2}^{n+1} \Big(1+\mathscr{O}_{\mathbb{P}^n}(e_i-1)y \Big).
\end{equation*}
Using the fact that $e_1 = 1$  and $\mathscr{O}_{\mathbb{P}^n}(D) \cong \mathscr{O}_{\mathbb{P}^n}(e_1 + \ldots + e_{n+1})$, we get
\begin{equation}\label{loghyper}
\big(\sum_{p\geq 0} \Omega_{\mathbb{P}^n}^p(\log D) y^p\big) \otimes \mathscr{O}_{\mathbb{P}^n}(-D) = \frac{\displaystyle\prod_{i=1}^{n+1}\Big(\mathscr{O}_{\mathbb{P}^n}(-e_i) + \mathscr{O}_{\mathbb{P}^n}(-1)y\Big)}{1+y}.
\end{equation}

We see that the difference class \eqref{diff} cannot be $0$ unless all the exponents are $1$, a situation which only happens when $\mathscr{A}$ is normal crossing.

Even though \eqref{diff} is nonzero in general, we can nevertheless check by hand that the Chern class formula \eqref{formula} is true when $\hat{\mathscr{A}}$ is free, as a corollary of our computation thus far. In fact, we just need to apply the generalised Todd class transformation $\textup{td}_{(1+y)*}$ to the classes \eqref{mchyper} and \eqref{loghyper} individually, and compare the values of the classes at $y=-1$. For example, 
\begin{equation*}
\begin{split}
\textup{ch}(mC_y([M(\mathscr{A}) \to \mathbb{P}^n])) \cdot \textup{td}(T\mathbb{P}^n) &= \frac{\displaystyle\prod_{i=1}^{n+1} \Big(1-e_i + (e_i+y)e^{-h}\Big)}{1+y} \cdot (\frac{h}{1-e^{-h}})^{n+1} \\
&= \frac{1}{1+y}\prod_{i=1}^{n+1}\frac{h(1 - e_i + e^{-h}(y+e_i))}{1-e^{-h}}
\end{split}
\end{equation*}
where $h = c_1(\mathscr{O}_{\mathbb{P}^n}(1))$. The normalisation of this class is 
\begin{equation*}
 \frac{1}{(1+y)^{n+1}}\prod_{i=1}^{n+1}\Big(\frac{1-e_i + e^{-(1+y)h}(y+e_i)}{1-e^{-(1+y)h}} \cdot h(1+y)\Big) = \prod_{i=1}^{n+1}\Big(\frac{1-e_i + e^{-(1+y)h}(y+e_i)}{1-e^{-(1+y)h}} \cdot h \Big).
\end{equation*}

Calculating the limit $y \to -1$, we obtain the expression
\begin{equation*}
\prod_{i=1}^{n+1} (1+(1-e_i)h)
\end{equation*}
in cohomology. So 
\begin{equation*}
c_*(1_{M(\mathscr{A})}) = \Big(\prod_{i=1}^{n+1} (1+(1-e_i)h) \Big)\cap [\mathbb{P}^n].
\end{equation*}
Performing the same procedure to \eqref{loghyper}, we will obtain 
\begin{equation*}
c(\textup{Der}_{\mathbb{P}^n}(-\log D)) \cap [\mathbb{P}^n] = \Big(\prod_{i=1}^{n+1} (1+(1-e_i)h) \Big)\cap [\mathbb{P}^n]
\end{equation*}
too. The calculation is omitted.

\bibliographystyle{alpha}
\bibliography{liaobib}

\begin{thebibliography}{CMNM02}

\bibitem[AAL77]{MR0463484}
William~A. Adkins, Aldo Andreotti, and J.~V. Leahy.
\newblock An analogue of {O}ka's theorem for weakly normal complex spaces.
\newblock {\em Pacific J. Math.}, 68(2):297--301, 1977.

\bibitem[Alu13]{MR3047491}
Paolo Aluffi.
\newblock Grothendieck classes and {C}hern classes of hyperplane arrangements.
\newblock {\em Int. Math. Res. Not. IMRN}, (8):1873--1900, 2013.

\bibitem[Bit04]{MR2059227}
Franziska Bittner.
\newblock The universal {E}uler characteristic for varieties of characteristic
  zero.
\newblock {\em Compos. Math.}, 140(4):1011--1032, 2004.

\bibitem[BSY10]{MR2646988}
Jean-Paul Brasselet, J{\"o}rg Sch{\"u}rmann, and Shoji Yokura.
\newblock Hirzebruch classes and motivic {C}hern classes for singular spaces.
\newblock {\em J. Topol. Anal.}, 2(1):1--55, 2010.

\bibitem[CMNM02]{MR1931962}
Francisco Calder{\'o}n-Moreno and Luis Narv{\'a}ez-Macarro.
\newblock The module {$\mathscr{D}f^s$} for locally quasi-homogeneous free
  divisors.
\newblock {\em Compositio Math.}, 134(1):59--74, 2002.

\bibitem[Fab15]{MR3319556}
Eleonore Faber.
\newblock Characterizing normal crossing hypersurfaces.
\newblock {\em Math. Ann.}, 361(3-4):995--1020, 2015.

\bibitem[FL85]{MR801033}
William Fulton and Serge Lang.
\newblock {\em Riemann-{R}och algebra}, volume 277 of {\em Grundlehren der
  Mathematischen Wissenschaften [Fundamental Principles of Mathematical
  Sciences]}.
\newblock Springer-Verlag, New York, 1985.

\bibitem[Ful98]{MR1644323}
William Fulton.
\newblock {\em Intersection theory}, volume~2 of {\em Ergebnisse der Mathematik
  und ihrer Grenzgebiete. 3. Folge. A Series of Modern Surveys in Mathematics
  [Results in Mathematics and Related Areas. 3rd Series. A Series of Modern
  Surveys in Mathematics]}.
\newblock Springer-Verlag, Berlin, second edition, 1998.

\bibitem[GS14]{MR3260143}
Michel Granger and Mathias Schulze.
\newblock Normal crossing properties of complex hypersurfaces via logarithmic
  residues.
\newblock {\em Compos. Math.}, 150(9):1607--1622, 2014.

\bibitem[Hir57]{MR0090850}
Heisuke Hironaka.
\newblock On the arithmetic genera and the effective genera of algebraic
  curves.
\newblock {\em Mem. Coll. Sci. Univ. Kyoto. Ser. A. Math.}, 30:177--195, 1957.

\bibitem[Kov11]{MR2784747}
S{\'a}ndor~J. Kov{\'a}cs.
\newblock Du {B}ois pairs and vanishing theorems.
\newblock {\em Kyoto J. Math.}, 51(1):47--69, 2011.

\bibitem[KS11]{MR2796408}
S{\'a}ndor~J. Kov{\'a}cs and Karl~E. Schwede.
\newblock Hodge theory meets the minimal model program: a survey of log
  canonical and {D}u {B}ois singularities.
\newblock In {\em Topology of stratified spaces}, volume~58 of {\em Math. Sci.
  Res. Inst. Publ.}, pages 51--94. Cambridge Univ. Press, Cambridge, 2011.

\bibitem[Lia]{liao}
Xia Liao.
\newblock Chern classes of logarithmic derivations for free divisors with
  jacobian ideal of linear type.
\newblock arXiv:1210.6079.

\bibitem[Lia12]{MR2928937}
Xia Liao.
\newblock Chern classes of logarithmic vector fields.
\newblock {\em J. Singul.}, 5:109--114, 2012.

\bibitem[LS]{sl}
Xia Liao and Mathias Schulze.
\newblock Quasihomogeneous free divisors with only normal crossings in
  codimension one.

\bibitem[LV81a]{MR618807}
John~V. Leahy and Marie~A. Vitulli.
\newblock Seminormal rings and weakly normal varieties.
\newblock {\em Nagoya Math. J.}, 82:27--56, 1981.

\bibitem[LV81b]{MR632650}
John~V. Leahy and Marie~A. Vitulli.
\newblock Weakly normal varieties: the multicross singularity and some
  vanishing theorems on local cohomology.
\newblock {\em Nagoya Math. J.}, 83:137--152, 1981.

\bibitem[MS15]{MR3417881}
Lauren{\c{t}}iu~G. Maxim and J{\"o}rg Sch{\"u}rmann.
\newblock Characteristic classes of singular toric varieties.
\newblock {\em Comm. Pure Appl. Math.}, 68(12):2177--2236, 2015.

\bibitem[OT92]{MR1217488}
Peter Orlik and Hiroaki Terao.
\newblock {\em Arrangements of hyperplanes}, volume 300 of {\em Grundlehren der
  Mathematischen Wissenschaften [Fundamental Principles of Mathematical
  Sciences]}.
\newblock Springer-Verlag, Berlin, 1992.

\bibitem[Pol15]{MR3319132}
Delphine Pol.
\newblock Logarithmic residues along plane curves.
\newblock {\em C. R. Math. Acad. Sci. Paris}, 353(4):345--349, 2015.

\bibitem[Sai80]{MR586450}
Kyoji Saito.
\newblock Theory of logarithmic differential forms and logarithmic vector
  fields.
\newblock {\em J. Fac. Sci. Univ. Tokyo Sect. IA Math.}, 27(2):265--291, 1980.

\bibitem[Yok98]{MR1677403}
Shoji Yokura.
\newblock A singular {R}iemann-{R}och for {H}irzebruch characteristics.
\newblock In {\em Singularities {S}ymposium---\L ojasiewicz 70 ({K}rak\'ow,
  1996; {W}arsaw, 1996)}, volume~44 of {\em Banach Center Publ.}, pages
  257--268. Polish Acad. Sci., Warsaw, 1998.

\end{thebibliography}
  
\end{document}